\documentclass[12pt,a4paper]{amsart}
\usepackage[utf8]{inputenc}
\usepackage{amsmath}
\usepackage{amsfonts}
\usepackage[colorlinks,linktocpage=true]{hyperref}
\hypersetup{colorlinks,linkcolor={blue},citecolor={blue},urlcolor={black}}

\usepackage{array,booktabs}
\usepackage{amssymb}

\theoremstyle{plain}
\newtheorem{theorem}{Theorem}[section]

\newtheorem{lemma}[theorem]{Lemma}

\theoremstyle{definition}

\theoremstyle{remark}
\newtheorem{remark}[theorem]{Remark}

\newtheorem*{remark3}{Acknowledgements}
\newcommand{\capacity}{\operatorname{cap}}
\newcommand{\diam}{\operatorname{diam}}
\DeclareMathOperator*{\esssup}{ess\,sup}
\DeclareMathOperator*{\essinf}{ess\,inf}
\def\C{\mathbb C}
\def\R{\mathbb R}
\def\Z{\mathbb Z}
\def\N{\mathbb N}
\def\H{\mathbb H}
\def\cA{\mathcal A}
\def\cI{\mathcal I}
\def\cJ{\mathcal J}
\def\cE{\mathcal E}

\def\cU{\mathcal U}
\def\cY{\mathcal Y}

\def\det{\operatorname{det}}
\def\dens{\operatorname{dens}}
\def\dist{\operatorname{dist}}
\def\length{\operatorname{length}}
\def\meas{\operatorname{meas}}
\def\card{\operatorname{card}}
\def\ann{\operatorname{ann}}
\def\id{\operatorname{id}}
\usepackage[showonlyrefs]{mathtools}
\mathtoolsset{showonlyrefs}

\numberwithin{equation}{section}
\title{Hausdorff dimension in quasiregular dynamics}
\author{Walter Bergweiler}
\address{Mathematisches Seminar, Christian--Albrechts--Uni\-versi\-t\"at zu Kiel,
Heinrich--Hecht--Platz 6,
24098 Kiel, Germany}
\email{bergweiler@math.uni-kiel.de}
\author{Athanasios Tsantaris}
\address{Department of Mathematics and Statistics, University of Helsinki, Pietari Kalmin katu 5, 00560 Helsinki, Finland}
\email{athanasios.tsantaris@helsinki.fi}
\makeatletter
\@namedef{subjclassname@2020}{%
  \textup{2020} Mathematics Subject Classification}
\makeatother
\subjclass[2020]{Primary 37F35; Secondary 30C65, 30D05,37F10}

\date{}

\begin{document}
\begin{abstract}
It is shown that the Hausdorff dimension of the fast escaping set of a 
quasiregular self-map of ${\mathbb R}^3$ can take any value in the interval $[1,3]$.
The Hausdorff dimension of the Julia set of such a map is estimated under some
growth condition.
\end{abstract}
\maketitle
\section{Introduction} \label{intro}
Quasiregular maps are a natural generalization of holomorphic maps to higher dimensions;
see \S~\ref{qr_maps} for their definition and basic properties.
We will be concerned with quasiregular self-maps of $\R^d$. These are the
analogues of entire maps $f\colon\C\to\C$. A  quasiregular map
$f\colon\R^d\to\R^d$  is said to be of \emph{polynomial type} if
$\lim_{|x|\to\infty}|f(x)|=\infty$ and of \emph{transcendental type} otherwise.
Here $|x|$ is the Euclidean norm of a point $x\in\R^d$.

The \emph{escaping set}
$\cI(f)$ of a quasiregular map $f\colon\R^d\to\R^d$ is defined by
\begin{equation} \label{0a}
\cI(f) = \{ x \in \R^d \colon |f^n(x)| \to \infty \}.
\end{equation}
Here $f^n$ is the $n$-th iterate of $f$. For transcendental entire functions 
the escaping set was introduced by Eremenko~\cite{Eremenko1989}
who, in particular, showed that it is always non-empty.
This result was extended to quasiregular maps of transcendental type in~\cite{BFLM}.

The proofs in~\cite{BFLM,Eremenko1989} actually show that not only the escaping
set $\cI(f)$ is non-empty, but that this in fact holds for a subset $\cA(f)$ of $\cI(f)$
called the \emph{fast escaping set}.
This set was first considered in~\cite{BH}.

In order to define it, recall that the \emph{maximum modulus} is given by
\begin{equation} \label{0b}
M(r,f)=\max_{|x|=r}|f(x)|.
\end{equation}
Let $M^n(r,f)$ denote the $n$-th iterate of $M(r,f)$
with respect to the first variable. We thus have 
$M^1(r,f) = M(r,f)$ and $M^n(r,f) = M(M^{n-1}(r,f),f)$ for $n\geq 2$.
For a quasiregular map $f$ of transcendental type
there exists $R_0\geq 0$ such that $M(R,f)>R$ for $R>R_0$.
For $R>R_0$ we then have $M^n(R,f)\to\infty$ as $n\to\infty$.
For such $R$ the fast escaping set is defined by 
\begin{equation} \label{0c}
\begin{aligned}
\cA(f) 
=  \{ x \in \R^d \colon & \text{there exists}\ L \in \N\ \text{such that}\ 
|f^{n+L}(x)| \geq M^n(R,f)\ \\ & \text{for all}\ n\in\N\}.
\end{aligned}
\end{equation}
This definition does not depend on $R$ as long as
$M^n(R,f)\to\infty$; see \cite[Theorem 2.2,~(b)]{Rippon2012} for transcendental entire functions
and~\cite[Proposition 3.1,~(i)]{BDF} for quasiregular maps. These papers also contain some
equivalent definitions of $\cA(f)$.

The Julia set $\cJ(f)$ of a transcendental entire function $f$ is defined as the set of 
all points where the iterates of $f$ do not form a normal family. 
Eremenko~\cite{Eremenko1989} showed that $\cJ(f)=\partial \cI(f)$.
As noted by Rippon and Stallard~\cite[p.~1125, Remark~1]{Rippon2005}, we also have
$\cJ(f)=\partial \cA(f)$.

For quasiregular mappings the set where the iterates are not normal does not
have the properties one would normally associate with a Julia set. 
Therefore, instead of using normality,
the Julia set $\cJ(f)$ was defined in~\cite{Bergweiler2013,BN} as the set of 
all $x\in\R^d$ such that $\R^d\setminus\bigcup_{k=1}^\infty f^k(U)$ has capacity
zero for every neighborhood $U$ of~$x$; see \S~\ref{qr_maps} for the definition of capacity.
As shown in these papers, the set $\cJ(f)$ defined this way shares many properties with Julia sets of entire
functions. However, there are examples~\cite[Example~7.3]{BN} where $\cJ(f)\neq \partial \cI(f)$.
On the other hand, we always have $\cJ(f)\subset \partial \cI(f)$; see~\cite[Theorem~1.3]{BN}.

By~\cite[Theorem~1.1]{BFN} we also have $\cJ(f)\subset\partial \cA(f)$.
Moreover, it was shown in~\cite[Theorem~1.2]{BFN} that under the additional hypothesis 
\begin{equation}\label{0d}
\liminf_{r\to\infty}\frac{\log\log M(r,f)}{\log \log r}=\infty
\end{equation}
we even have $\cJ(f)=\partial \cA(f)$.
Thus for quasiregular maps the fast escaping set is perhaps even more relevant
for the dynamics than the escaping set.

There is a large literature concerned with the Hausdorff dimension of the above sets
for transcendental entire functions; see~\cite{Kotus2008} and~\cite{Stallard2008}
for surveys.
We will describe some of these results, denoting by
$\dim X$ the Hausdorff dimension of a subset $X$ of $\R^d$.

Baker~\cite{Baker1975} proved that the Julia set of 
a transcendental entire function $f$ contains continua.  Thus $\dim \cJ(f)\geq 1$.
On the other hand, for all $\rho\in[1,2]$ there exists a transcendental entire
function $f$ satisfying $\dim \cJ(f)=\rho$. Here the case $\rho\in (1,2)$ is due to 
Stallard~\cite[Theorem~1.1]{Stallard2000} while the case $\rho=1$ is due to 
Bishop~\cite{Bishop2018}. Examples with $\dim \cJ(f)=2$ had been known already before.
Perhaps the simplest ones are those with $\cJ(f)=\C$.
The first example of such a function $f$ was given by Baker~\cite{Baker1970}.

Rippon and Stallard~\cite[Theorem~1]{Rippon2005} showed that every component of 
$\cA(f)$ is unbounded. While a conjecture of Eremenko~\cite{Eremenko1989}
saying that every component of $\cI(f)$ is unbounded was recently 
disproved~\cite{Marti-Pete2022}, the result by Rippon and Stallard
implies of course that $\cI(f)$ has at 
least one unbounded component. In particular,
$\dim \cI(f)\geq \dim \cA(f)\geq 1$.
On the other hand, it was shown in~\cite[Corollary~1.2]{Rempe2010} that for all 
$\rho\in[1,2]$ there exists a transcendental entire function $f$ such that 
$\dim \cI(f)=\rho$. 
The main contribution of~\cite{Rempe2010} is the case $\rho=1$. The case
$\rho\in(1,2)$ is deduced from results in~\cite{Stallard2000} and~\cite{BKS}.
Inspection of the proofs in~\cite{BKS} shows that we actually may achieve
$\dim \cA(f)=\rho$ for given $\rho\in(1,2)$. 
For $\rho=2$ such an example is given by the exponential function;
this follows from the arguments in~\cite[Theorem~1.2]{McMullen1987}.
Further examples are given in~\cite{Sixsmith2015}.
Altogether we see that for all $\rho\in[1,2]$ there exists a transcendental
entire function $f$ such that $\dim \cA(f)=\rho$.

It is the main purpose of this paper to extend this result
to quasiregular maps of transcendental type.
We will restrict to the case that $d=3$. The reason is that this allows to use a
construction of Nicks and Sixsmith~\cite{Nicks2018}.

\begin{theorem}\label{theorem_dimA}
For all $\rho\in [1,3]$ there exists a quasiregular map 
$f\colon \mathbb{R}^3\to \mathbb{R}^3$ of transcendental type
such that $\dim \mathcal{A}(f)=\rho$.
\end{theorem}

As already mentioned, we have $\dim \cI(f)\geq 1$ for a quasiregular map
$f\colon \R^d\to\R^d$ of transcendental type.
We do not know whether this bound is sharp. It seems plausible that we have 
$\dim \cI(f)\geq d-1$ for such maps; see Remark~\ref{dimI} below.

Perhaps we also have $\dim \cJ(f)\geq d-1$.
However, no positive lower bound is known for $\dim \cJ(f)$.
It was shown in~\cite[Theorem~1.7]{BN} that $\dim \cJ(f)>0$ if $f$ is locally Lipschitz continuous. 
We obtain a lower bound under the growth condition~\eqref{0d}.

\begin{theorem}\label{theorem_dimJ}
Let $f\colon\mathbb{R}^d\to\mathbb{R}^d$ be a quasiregular map satisfying \eqref{0d}. 
Then $\dim \mathcal{J}(f)\geq 1$.
\end{theorem}
The structure of the rest of this paper is as follows: In Sections \ref{prelim} and \ref{proof1} we recall some definitions and some preliminaries. In Section  \ref{lower_bound} we prove Theorem \ref{theorem_dimA} and finally in Section \ref{proof_theorem_dimJ} we prove Theorem \ref{theorem_dimJ}.

We conclude the introduction with a remark on quasiregular self-maps of~$\overline{\R^d}$.
\begin{remark}
Fletcher and Vellis \cite{Fletcher2021} studied the more general question of when can a Cantor set be the
Julia set of a uniformly quasiregular map. It follows
from \cite[Theorem 1.2]{Fletcher2021} that for any $\rho\in(0,d-1)$ there is a uniformly quasiregular
map $f\colon\overline{\R^d}\to\overline{\R^d}$ such that $\dim \mathcal{J}(f)=\rho$. 
	 
Moreover, by \cite[Corollary 1.4]{Fletcher2021} we can find a uniformly quasiregular map
$f\colon\overline{\R^3}\to\overline{\R^3}$ whose Julia set is a Cantor set which is
defined as the limit set of an iterated function system of similarities acting on $[0,1]^3$.
By choosing the right similarities we can make this set have any Hausdorff dimension in $(0,3)$.

\end{remark}

\begin{remark3}
Part of this work was done while the second author was visiting the University of Kiel which he would like to thank for its hospitality. He would like to thank the London Mathematical Society and the University of Kiel for funding this visit. Also, the second author would like to thank Dan Nicks for many interesting discussions. 
\end{remark3}

\section{Basic results about quasiregular maps and Hausdorff dimension} \label{prelim}
\subsection{Quasiregular maps}  \label{qr_maps} 
We only sketch the definition and basic properties of quasiregular maps.
For a detailed treatment we refer to Rickman's book~\cite{Rickman1993}.

Let $\Omega\subset\R^d$ be a domain and let $f\colon \Omega \to \R^{d}$ be a map.
If $f$ is differentiable at a point $x\in\Omega$, let $Df(x)$ be the derivative,
$|Df(x)|=\sup_{|h|=1}|Df(x)(h)|$ its norm and $J_f(x)=\det J_f(x)$ the
 Jacobian determinant.
We will also use the notation $\ell(Df(x))=\inf_{|h|=1}|Df(x)(h)|$.
The map  $f$ is called \emph{quasiregular} if it is continuous, belongs to the Sobolev
space $W^{1}_{d, loc}(\Omega)$, and there exists $K\geq 1$ such that 
\begin{equation} \label{0e}
|Df(x)|^{d} \leq K J_{f}(x)
\end{equation}
almost everywhere in~$\Omega$. The smallest such $K$ is called the \emph{outer dilatation} $K_O(f)$.
The last condition is equivalent to the existence of a constant $K'>1$
such that 
\begin{equation}\label{0f}
J_f(x)\leq K'\ell(Df(x))^d
\end{equation}
almost everywhere in~$\Omega$. The smallest such $K'$ is called \emph{inner dilatation} $K_I(f)$ of~$f$.
An injective quasiregular map is called \emph{quasiconformal}.

For a survey of results concerning the dynamics or quasiregular maps
we refer to~\cite{Bergweiler2010a}.

In the proof of Theorem \ref{theorem_dimJ} we also need the concept of the modulus of a path family.
We only recall here the relevant definitions, for a more thorough discussion we refer to \cite{Rickman1993}.

Let $\Gamma$ be a family of paths in $\mathbb{R}^d$, with $d\geq 2$.
We denote by $\mathcal{F}(\Gamma)$ the set of all Borel functions $\rho\colon\mathbb{R}^d\to [0,\infty]$
such that $\int_\gamma\rho\; ds\geq 1$ for every locally rectifiable path $\gamma\in \Gamma$. 
Then the \emph{modulus} $M(\Gamma)$ of $\Gamma$ is defined as
\begin{equation}
	M(\Gamma)=\inf_{\rho\in \mathcal{F}(\Gamma)}\int_{\mathbb{R}^d}\rho^d dm,
\end{equation}
where $m$ is the Lebesgue measure. If $E,F,D\subset \mathbb{R}^d$, we will denote
by $\Delta\!\left(E,F;D\right)$ the family of paths
$\gamma\colon [0,1]\to \mathbb{R}^d$ such one of the endpoints $\gamma(0)$ and $\gamma(1)$
belongs to $E$, the other one to $F$, while  $\gamma(t)\in D$ for all $t\in(0,1)$.

Let $\Omega\subset\R^d$ be open and let $C$ be a compact subset of~$\Omega$.
Then the pair $(\Omega,C)$ is called a \emph{condenser} and
its \emph{capacity} $\capacity (\Omega,C)$ is defined by
\[
\capacity (\Omega,C)=\inf_u\int_\Omega\left|\nabla u\right|^d dm,
\]
where the infimum is taken over all non-negative functions
$u\in C^\infty_0(\Omega)$ satisfying $u(x)\geq 1$ for $x\in C$.
If $\capacity (\Omega,C)=0$ for some bounded open set $\Omega$ containing~$C$, then,
by \cite[Lemma III.2.2]{Rickman1993},
we have $\capacity (\Omega',C)=0$ for every bounded open set $\Omega'$ containing~$C$.
In this case we say that $C$ is of  \emph{capacity zero}, while
otherwise $C$ has \emph{positive capacity}.
A closed subset $C$ of $\R^d$ is said to have capacity
zero if this is the case for all compact subsets of~$C$.

The connection between capacity and the modulus of a path family is given by~\cite[Proposition II.10.2]{Rickman1993}
\[
\capacity (\Omega,C)=M(\Delta(C,\partial \Omega;\Omega)).
\]

Zorich~\cite[p.~400]{Zorich1967} introduced an 
important example of a quasiregular self-map of $\R^3$
which can be considered as a $3$-dimensional analogue of the exponential 
function.
In fact, the construction is quite flexible and 
also works in $\R^d$ for all $d\geq 2$;
see~\cite[\S~8.1]{Martio1975}.
We follow~\cite[\S~6.5.4]{Iwaniec2001} in the definition of a Zorich map 
and consider the cube
\begin{equation}\label{0g}
Q:=\left\{x\in \R^{d-1}\colon \|x\|_\infty \leq 1\right\}=[-1,1]^{d-1}
\end{equation}
and the upper hemisphere
\begin{equation}\label{0h}
U:=\left\{x\in \R^{d}\colon \|x\|_2 = 1, x_d\geq 0 \right\}.
\end{equation}
Let $h\colon Q\to U$ be a bi-Lipschitz map and define
\begin{equation}\label{0i}
Z\colon Q\times \R \to \R^d,\quad
Z(x_1,\dots,x_d)=e^{x_d}h(x_1,\dots,x_{d-1}).
\end{equation}
The map $Z$ is then extended to a map $Z\colon \R^d\to\R^d$ by repeated reflection at 
hyperplanes.
Many results about the dynamics of exponential functions have been extended to
Zorich maps~\cite{Bergweiler2010,BD,Comduehr2019,Tsantaris2020,Tsantaris2021}.
In particular, we note that it follows from the arguments given in~\cite[\S~4]{Bergweiler2010}
that for suitable $c\in\R^3$ and $f\colon\R^3\to\R^3$, $f(x)=Z(x)-c$, we have 
$\dim\cA(f)=3$. Thus we may restrict in our proof
of Theorem~\ref{theorem_dimA} to the case that $1\leq \rho<3$.
We will use a modification of the Zorich map introduced by Nicks and
Sixsmith~\cite{Nicks2018}; see~\S~\ref{def_f}.

\subsection{Hausdorff dimension}  \label{h_dim} 
For the definition and a thorough treatment of Hausdorff dimension we refer 
to the book by Falconer~\cite{Falconer1990}.
To obtain a lower bound for  the Hausdorff dimension we will use the following
result known as the mass distribution principle~\cite[Proposition 4.9]{Falconer1990}.
Here and in the following we use the notation $B(x,r)$ for the open ball of
radius $r$ around a point $x\in\R^d$. Sometimes we will emphasize the dimension
by writing $B_d(x,r)$ instead of $B(x,r)$.
The closed ball is denoted by $\overline{B}(x,r)$.
\begin{lemma} \label{lemma-frostman}
Let $E$ be a compact subset of $\R^d$. Suppose that there exist a
probability measure $\mu$ supported on $E$ and positive constants
$c$, $r_0$ and $\rho$ such that
\begin{equation}\label{2c}
\mu(B(x,r))\leq cr^\rho
\end{equation}
for each $x\in E$ and each $r \in (0,r_0)$.
Then $\dim E\geq \rho$.
\end{lemma}

To obtain an upper bound for  the Hausdorff dimension we will use 
the following result~\cite[Lemma 5.2]{Bergweiler2010}. It is a simple 
consequence of a standard covering result~\cite[Covering Lemma~4.8]{Falconer1990}.
\begin{lemma} \label{lemma10}
Let $E\subset \R^d$ and $\rho>0$.
Suppose that for all $x\in E$ and $\delta >0$ there exist
$r(x)\in (0,1)$, $d(x)\in (0,\delta)$ and $N(x)\in \N$ satisfying
\begin{equation}\label{2d}
N(x)d(x)^\rho\leq r(x)^n
\end{equation}
 such that $B(x,r(x))\cap E$ can be
covered by $N(x)$ sets of diameter at most $d(x)$. Then
$\dim E\leq\rho$.
\end{lemma}
In~\cite{Bergweiler2010} it is assumed in addition that $E$ is bounded, 
but this hypothesis can be omitted since the Hausdorff dimension of an unbounded set
is the supremum of the Hausdorff dimensions of its bounded subsets. Note that in~\cite{Bergweiler2010} it was assumed that $\rho>1$ but the proof also works for $\rho>0$.

Let $\meas X$ denote the Lebesgue measure of a measurable set~$X$. We will need the concept of
the \emph{density} $\dens(A,B)$ of a set $A$ in a set~$B$,
where $A,B\subset \R^d$ are measurable with $\meas B>0$.  It is defined by
\begin{equation}\label{2a}
\dens(A,B)=\frac{\meas(A\cap B)}{\meas B}.
\end{equation}

To estimate how the density changes under a quasiconformal map we will use the 
following result.
We omit its simple proof.
\begin{lemma} \label{lemma-dens}
Let $f\colon \Omega\to\R^d$ be quasiregular and let $A,B\subset \Omega$ be measurable 
with $\meas B>0$. Suppose that $f$ is injective on~$B$.
Then 
\begin{equation}\label{2b}
\left(\frac{\essinf_{x\in B}\ell(Df(x))}{\esssup_{x\in B}|Df(x)|}\right)^d
\leq 
\frac{\dens(f(A),f(B))}{\dens(A,B)}
\leq 
\left(\frac{\esssup_{x\in B}|Df(x)|}{\essinf_{x\in B}\ell(Df(x))}\right)^d .
\end{equation}
\end{lemma}

\section{Preliminaries for the proof of Theorem~\ref{theorem_dimA}} \label{proof1}
\subsection{Definition of \texorpdfstring{$f$}{f}} \label{def_f}
Nicks and Sixsmith~\cite[Section~5]{Nicks2018}
 considered a variant of the Zorich map where 
$h$ maps $[-1,1]^2$ not to the upper hemisphere but 
to the upper faces of a square based pyramid.
Specifically, they worked with 
\begin{equation}\label{1a}
h(x_1,x_2)=(x_1,x_2,1-\max\{|x_1|,|x_2|\}).
\end{equation}
The map $Z$ is again defined by 
\begin{equation}\label{1b}
Z(x)=e^{x_3}h(x_1,x_2)
\quad\text{for}\ |x_1|\leq 1\ \text{and}\ |x_2|\leq 1,
\end{equation}
and extended to a map $Z\colon\R^3\to \R^3\setminus\{0\}$ by
reflections.
They used this map $Z$ to construct a quasiregular map $G\colon\R^3\to\R^3$ 
which is equal to the identity in a half-space. 

For $t\in\R$ we put $\H_{\geq t}=\{(x_1,x_2,x_3)\in\R^3\colon x_3\geq t\}$.
The half-spaces $\H_{\leq t}$, $\H_{> t}$ and $\H_{< t}$ are defined 
analogously.
The quasiregular map $G$ constructed by 
Nicks and Sixsmith~\cite[Section~6]{Nicks2018} has the property
that there exists $L>0$ such that 
\begin{equation}\label{1c}
G(x)=
\begin{cases}
x+Z(x) &\text{for}\ x\in\H_{\geq L},\\
x &\text{for}\ x\in\H_{\leq 0}.
\end{cases}
\end{equation}
The difficult part in the construction is, of course, to define $G$ 
in the remaining domain
$\{x\in\R^3\colon 0<x_3<L\}$. Here we only note that the construction
yields that there exists a constant $C$ such that 
\begin{equation}\label{1c1}
|G(x)-x|\leq C
\quad\text{for}\ x\in\H_{\leq L}.
\end{equation}

We will also consider the maps 
$\varphi\colon \R^3\to\R^3$,
\begin{equation}\label{1c2}
\varphi(x_1,x_2,x_3)=(x_1,x_2,x_3-|x_1|-|x_2|)
\end{equation}
and $H\colon \R^3\to\R^3$,
\begin{equation}\label{1c3}
H=Z\circ\varphi .
\end{equation}
The map $\varphi$ is quasiconformal and hence $H$ is quasiregular.

Considering $H$ instead of $Z$ has the advantage that while
$Z$ is bounded in a half-space, $H$ is bounded in a larger domain.
In fact, with 
\begin{equation}\label{1c4}
\Omega=\{x\in\R^3\colon x_3> |x_1|+|x_2|\}
\end{equation}
we have 
\begin{equation}\label{1c5}
|H(x)|\leq \max_{y\in[-1,1]^2}|h(y)|=\sqrt{2}
\quad\text{for}\ x\in\R^3\setminus\Omega.
\end{equation}

For $s=(s_1,s_2)\in\Z^2$ we consider the beam
\begin{equation}\label{1d}
\begin{aligned}
T(s)
&=[2s_1-1,2s_1+1]\times [2s_2-1,2s_2+1]\times\R
\\ &
=\{x\in\R^3\colon |x_j-2s_j|\leq 1\ \text{for}\ j\in\{1,2\}\}.
\end{aligned}
\end{equation}
By construction, $Z$ is injective in $T(s)$.
Since $\varphi$ maps $T(s)$ bijectively onto itself, 
$H$ is also injective in $T(s)$.
The definition of $Z$ also yields that there exists $R\geq L$ such 
that if $x,y\in T(s)\cap\H_{\geq R}$ and $x\neq y$, then $|Z(x)-Z(y)|>|x-y|$ and hence
\begin{equation}\label{1d1}
|G(x)-G(y)|\geq |Z(x)-Z(y)|-|x-y|>0 .
\end{equation}
We deduce that $G$ is injective in $T(s)\cap\H_{\geq R}$.

We also note that if $s=(s_1,s_2)\in\Z^2$ and $s_1+s_2$ is even, then
both $Z$ and $G$ map  $T(s)\cap\H_{\geq R}$ into
$\H_{\geq 0}$, provided $R$ is large enough.

Since  $\varphi$ leaves $\partial T(s)$ invariant while
$Z$ maps $\partial T(s)$ to the plane $\{x\in\R^3\colon x_3=0\}$ 
we find that 
\begin{equation}\label{1e}
G(H(x))=H(x)\quad\text{for}\ x\in\partial T(s). 
\end{equation}
We will define $f$ by putting $f(x)=G(H(x))$ in some beams while
$f(x)=H(x)$ in the other beams. 

In order to define this partition of the set of beams, 
we will use the following result.
\begin{lemma} \label{lemma11}
Let $0<\beta<1$.
Then there exists a subset $S$ of $\Z^2$ with $(0,0)\in S$ such that 
$s_1+s_2$ is even for $s=(s_1,s_2)\in S$ and such that
the following two conditions hold:
\begin{itemize} 
\item[$(a)$] If $s,s'\in S$ and $s\neq s'$, then $|s-s'|\geq |s|^\beta+|s'|^\beta$.
\item[$(b)$] For all $x\in\R^2$ there exists $s\in S$ such that
$|x-s|\leq |s|^\beta+(|x|+1)^\beta +1$.
\end{itemize} 
\end{lemma} 
\begin{proof} 
Let $\Sigma$ be the set of all $s=(s_1,s_2)\in\Z^2$ such that $s_1+s_2$ is even.
We put $s_0=(0,0)$ and define a sequence $(s_k)$ recursively as follows.
Assuming that $s_0,s_1,\dots,s_{k-1}$ have been defined, let $A_k$ denote the 
set of all $s\in\Sigma$ such that $|s-s_j|\geq |s|^\beta+|s_j|^\beta$ for 
all $j\in\{0,1,\dots,k-1\}$. Since $\beta<1$ this clearly holds if $|s|$
is sufficiently large so that $A_k\neq\emptyset$. We then choose $s_k\in A_k$ such that 
$|s_k|=\min_{s\in A_k}|s|$. Finally we put $S=\{s_k\colon k\geq 0\}$.

It follows from the construction that $S$ satisfies $(a)$. To prove~$(b)$,
let $x\in\R^2$. Then there exists $s'\in\Sigma$ with $|x-s'|<1$. 
It follows from the construction of $S$ that there exists $s\in S$ 
with $|s-s'|\leq |s|^\beta+|s'|^\beta$. In fact, otherwise one would have 
to choose $s_k=s'$ at some point of the construction. It now follows 
that $|x-s|\leq |x-s'|+|s'-s|< 1+|s|^\beta+|s'|^\beta\leq 1+|s|^\beta+(|x|+1)^\beta$.
\end{proof} 

As noted at the end of~\S~\ref{qr_maps}, we may restrict in our proof of 
Theorem~\ref{theorem_dimA} to the case that $1\leq \rho<3$.
If $\rho>1$, we put 
\begin{equation}\label{def-beta}
\beta=\frac{3-\rho}{2}
\end{equation}
 and choose $S$ according to Lemma~\ref{lemma11}.
If $\rho=1$, we put $S=\{(0,0)\}$. We now put
\begin{equation}\label{1f}
T=\bigcup_{s\in S} T(s) .
\end{equation}
The map $f$ is then defined by
\begin{equation}\label{1g}
f(x)=
\begin{cases}
G(H(x)) &\text{if}\ x\in T,\\
H(x) &\text{if}\ x\notin T.
\end{cases}
\end{equation}
Note that by~\eqref{1e} the map $f$ is quasiregular.

We will need the following estimates for the number of points of $S$ in certain
disks.
\begin{lemma} \label{lemma12}
Let $\beta$ and $S$ be as in Lemma~\ref{lemma11}.
There exists $r_0>0$ and $\alpha_0>0$ such that if $r>r_0$ and $x\in\R^2$, then 
\begin{equation}\label{7a0}
\card(S\cap B(x,r)) \leq \alpha_0 r^{2-2\beta}.
\end{equation}
Moreover,
\begin{equation}\label{7a}
\frac{1}{64} |x|^{-2\beta}r^2
\leq \card(S\cap B(x,r))
\leq 6 |x|^{-2\beta}r^2
\quad\text{for}\ 8|x|^\beta\leq r\leq \frac12 |x|
\end{equation}
and
\begin{equation}\label{7b}
\card(S\cap B(x,r))\leq 324
\quad\text{for}\ r<8|x|^\beta.
\end{equation}
\end{lemma} 
It follows from~\eqref{7b} that 
\begin{equation}\label{7a1}
\card(S\cap B(x,r))
\leq 324 |x|^{-2\beta}r^2
\quad\text{for}\ |x|^\beta\leq r\leq 8|x|^\beta  .
\end{equation}
We thus have an estimate of the same type as~\eqref{7a} for a larger range
of values of~$r$.
\begin{proof}[Proof of Lemma~\ref{lemma12}] 
We begin by proving~\eqref{7a} and~\eqref{7b}, from which we will deduce~\eqref{7a0}
afterwards.
So let $x$ and $r$ be such that one the conditions in~\eqref{7a} or~\eqref{7b} is satisfied.
If $r$ is large, then so is $|x|$. 
Choosing $r_0$ large we can thus achieve  that $8|x|^\beta<|x|/2$ for any $x$ 
satisfying one of the conditions in~\eqref{7a} or~\eqref{7b}.
Thus we have $r\leq |x|/2$ in both cases. Hence
\begin{equation}\label{7d}
\frac12 |x|\leq |y|\leq \frac32 |x|
\quad\text{for}\ y\in B(x,r).
\end{equation}
By the definition of $S$ we have 
\begin{equation}\label{7c}
B(s,|s|^\beta)\cap B(s',|s'|^\beta)=\emptyset 
\quad\text{if}\ s,s'\in S.
\end{equation}
Combining the last two equations we find that 
\begin{equation}\label{7e}
B(s,2^{-\beta}|x|^\beta)\cap B(s',2^{-\beta}|x|^\beta)=\emptyset 
\quad\text{if}\ s,s'\in S\cap B(x,r).
\end{equation}
We also have 
\begin{equation}\label{7f}
\bigcup_{s\in S\cap B(x,r)} B(s,2^{-\beta}|x|^\beta)
\subset  B(x,r+2^{-\beta}|x|^\beta) .
\end{equation}
If $8|x|^\beta\leq r$ we have $r+2^{-\beta}|x|^\beta\leq r+2^{-\beta}r/8\leq 9r/8$.
Hence 
\begin{equation}\label{7f1}
\bigcup_{s\in S\cap B(x,r)} B(s,2^{-\beta}|x|^\beta)
\subset B\!\left(x,\frac98 r \right)
\end{equation}
in this case.
Taking the measure on both sides and noting that by~\eqref{7e} the union on the left
hand side is disjoint, we deduce that 
\begin{equation}\label{7g}
\pi 2^{-2\beta}|x|^{2\beta} \card(S\cap B(x,r))
\leq \pi \frac{81}{64}r^2
\end{equation}
so that
\begin{equation}\label{7h}
\card(S\cap B(x,r)) \leq 2^{2\beta}\frac{81}{64}|x|^{-2\beta} r^2 .
\end{equation}
This proves the right inequality in~\eqref{7a}.

If $r< 8|x|^\beta$, then $r+2^{-\beta}|x|^\beta\leq 9|x|^\beta$.
Then~\eqref{7f} yields that 
\begin{equation}\label{7i}
\bigcup_{s\in S\cap B(x,r)} B(s,2^{-\beta}|x|^\beta)
\subset B\!\left(x,9 |x|^\beta \right) .
\end{equation}
Taking the measure on both sides now yields that 
\begin{equation}\label{7j}
 \card(S\cap B(x,r))\leq 2^{2\beta} 81
\end{equation}
and hence~\eqref{7b}.

To prove the left inequality in~\eqref{7a}, let $y\in B(x,r/2)$.
By the definition of~$S$, there exists $s=s(y)\in S$ such that 
$|y-s|\leq |s|^\beta+(|y|+1)^\beta +1$.
By~\eqref{7d} we have 
\begin{equation}\label{7k}
|y-s|\leq |s|^\beta+\left(\frac32 |x|+1\right)^\beta +1 .
\end{equation}
First we show that this implies that $|s|\leq 2|x|$ if $r_0$ is sufficiently large.
To do so, suppose that $|s|>2|x|$.
Using~\eqref{7d} again we find that 
\begin{equation}\label{7l}
|y-s|\geq |s|-|y|\geq |s|-\frac32 |x|> |s|-\frac34|s|=\frac14|s|.
\end{equation}
Together with~\eqref{7k} this yields that
\begin{equation}\label{7m}
\frac14|s|\leq |s|^\beta+\left(\frac32 |x|+1\right)^\beta +1
< |s|^\beta+\left(\frac34 |s|+1\right)^\beta +1,
\end{equation}
which is a contradiction if $r_0$ is large,
since then $|x|$ and hence $|s|$ are also large.
Hence we have $|s|\leq 2|x|$.
It thus follows from~\eqref{7k} for large $r_0$ that
\begin{equation}\label{7n}
|y-s|\leq 2^\beta |x|^\beta + \left(\frac32 |x|+1\right)^\beta +1< 4|x|^\beta .
\end{equation}
Since we are assuming that $8|x|^\beta<r$ this implies that $|y-s|<r/2$.
Hence $s\in B(y,r/2)\subset B(x,r)$.
Altogether we have thus shown that
 for all $y\in B(x,r/2)$ there exists $s=s(y)\in S\cap B(x,r)$ satisfying~\eqref{7n}.
Hence
\begin{equation}\label{7o}
B\!\left(x,\frac{r}{2}\right) \subset 
\bigcup_{s\in S\cap B(x,r)} B(s,4|x|^\beta).
\end{equation}
Considering the measure on both sides we obtain
\begin{equation}\label{7p}
\pi \frac{r^2}{4} \leq \pi 16 |x|^{2\beta} \card(S\cap B(x,r))
\end{equation}
and hence the left inequality of~\eqref{7a}.

It follows from~\eqref{7a} and~\eqref{7b} that~\eqref{7a0} holds if $|x|\geq 2r$.
To prove~\eqref{7a0} if $|x|\leq 2r$ we introduce the notation
$\ann(r_1,r_2):=\{x\colon r_1\leq |x|\leq r_2\}$ for an annulus.
It is easy to see that $\ann(2r/3,4r/3)$ can be covered by $8$ disks
of radius $r/2$ and center on $\partial B(0,r)$. 
It thus follows from~\eqref{7a} that this annulus contains at most 
$12r^{2-2\beta}$ points of~$S$.
Replacing $4r/3$ by $r$ we find that 
\begin{equation}\label{7q}
\card\!\left(S\cap \ann\!\left(\frac{r}{2},r\right) \right)\leq 12 r^{2-2\beta}
\end{equation}
if $r\geq r_0$.
It follows that 
\begin{equation}\label{7r}
\card\!\left(S\cap \ann\!\left(\frac{r}{2^{k+1}},\frac{r}{2^{k}}\right) \right)\leq 12 
\frac{r^{2-2\beta}}{2^{(2-2\beta)k}}
\end{equation}
for $k\in\N$ as long as $r/2^{k}\geq r_0$.
This implies that $\card(S\cap B(0,r))\leq \alpha_1 r^{2-2\beta}$ with some constant
$\alpha_1$ for $r\geq r_0$.
For $|x|\leq 2r$ we have $B(x,r)\subset B(0,3r)$ and thus~\eqref{7a0} follows 
with $\alpha_0=3^{2-2\beta}\alpha_1$.
\end{proof} 
\subsection{Characterization of \texorpdfstring{$\mathcal{A}(f)$}{A(f)}} \label{char_A}
For $m\in\N$ we put 
\begin{equation}\label{3b}
F_m=
\begin{cases}
H & \text{if}\ k \ \text{is odd},\\
G & \text{if}\ k \ \text{is even},
\end{cases}
\end{equation}
and 
\begin{equation}\label{3b1}
h_m=F_m\circ F_{m-1}\circ\dots\circ F_1,
\end{equation}
with $h_0=\id$. Then
\begin{equation}\label{3c}
h_{2n}=f^n
\quad \text{and} \quad
h_{2n+1}=
H\circ f^n. 
\end{equation}
It follows from the definition of $H$ that
\begin{equation}\label{8a}
|H(x)|\leq |Z(x)|\leq \exp|x| 
\quad\text{for all}\ x\in\R^3
\end{equation}
while the definition of $G$ and~\eqref{1c1} yield that
\begin{equation}\label{8b}
|G(x)|\leq |Z(x)|+|x|+C\leq \exp|x|+|x|+C
\quad\text{for all}\ x\in\R^3 .
\end{equation}
For large $R$ we thus have 
\begin{equation}\label{8c}
M(R,F_j)\leq \exp R + R +C \leq \frac12 \exp(2R) 
\end{equation}
for all $j$ and hence 
\begin{equation}\label{8d}
M(R,h_m)\leq \frac12 \exp^m(2R) .
\end{equation}
On the other hand, 
\begin{equation}\label{8e}
|H(0,0,R)|= |Z(0,0,R)|=\exp R
\end{equation}
for all $R>0$ and 
\begin{equation}\label{8f}
|G(0,0,R)|= |Z(0,0,R)+(0,0,R)|=\exp R+R\geq \exp R
\end{equation}
for $R\geq  L$.
Thus 
\begin{equation}\label{8g}
M(R,h_m)\geq \exp^m(R) 
\end{equation}
for $R\geq  L$.
In particular, it follows from~\eqref{3c}, \eqref{8d} and~\eqref{8g} that
\begin{equation}\label{8h}
\exp^{2n}(R)\leq M^{n}(R,f)\leq \frac12 \exp^{2n}(2R) 
\end{equation}
for large  $R$.
Recalling that the definition of $\cA(f)$ is independent of the choice of 
$R$ as long as $R> R_0$,  it follows from the above considerations and the 
definition of $f$ that for all $R>0$ we have 
\begin{equation}\label{8i}
\begin{aligned}
\cA(f) = \{ x \in \R^3 \colon  
& \text{there exists}\ L \in \N \ \text{such that}\ 
|f^{n+L}(x)| \geq \exp^{2n}(R) \ 
\\ &
\text{and}\ f^{n+L}(x)\in T\ \text{for all}\  n\in \mathbb{N}\} .
\end{aligned}
\end{equation}
This easily yields that 
\begin{equation}\label{8j}
\begin{aligned}
\cA(f) = \{ x \in \R^3 \colon  
& \text{there exists}\ L \in \N \ \text{such that}\ 
|h_{m+L}(x)| \geq \exp^{m}(R) \ 
\\ &
\text{for all}\  m\in \mathbb{N}
\ \text{and}\ h_{m+L}(x)\in T\ \text{if}\ m+L \ \text{is even}\} .
\end{aligned}
\end{equation}

The following result can easily be deduced from~\eqref{8j}.
\begin{lemma} \label{lemma13}
Let $x\in\cA(f)$ and $C>0$.
Then $|h_m(x)|\geq |h_{m-1}(x)|^C$ for all large~$m$.
Moreover, $h_m(x)_3\to\infty$ as $m\to\infty$.
\end{lemma} 
It follows from Lemma~\ref{lemma13} that given $C>0$ there exists $M\in\N$
such that 
\begin{equation}\label{8k}
\prod_{j=M}^{m-1} |h_{j}(x)| \leq 
\prod_{j=M}^{m-1} |h_{m}(x)|^{C^{j-m}} =|h_{m}(x)|^{\gamma_m}
\end{equation}
for $m> M$, with
\begin{equation}\label{8l}
\gamma_m=\sum_{j=M}^{m-1}C^{j-m} = \sum_{k=1}^{m-M}C^{-k} \leq \frac{1}{C-1}.
\end{equation}
It follows that if $x\in \cA(f)$ and  $\delta>0$, then
\begin{equation}\label{8m}
\prod_{j=1}^{m-1} |h_{j}(x)| \leq |h_{m}(x)|^\delta 
\end{equation}
for all large $m$. It also follows from Lemma \ref{lemma13} that if $x\in\cA(f)$ and $C>0$,
then there exists $\delta>0$ such that 
\begin{equation}\label{eq8m}
|h_m(x)|\geq \exp( \delta C^m)
\end{equation}
for all large $m$.
\subsection{Estimates for the derivatives of \texorpdfstring{$H$}{H} and \texorpdfstring{$G$}{G}} \label{derivatives}
The derivative of a branch $\Lambda$ of the inverse of the ``classical'' Zorich
map $Z$ introduced in \S~\ref{qr_maps} was estimated
in~\cite[inequalities (2.6) and (2.7)]{Bergweiler2010}. It is easy 
to see that the inequalities obtained there also hold for the modified 
Zorich map $Z$ introduced by Nicks and Sixsmith.
Thus there exist constants $\alpha_1,\alpha_2$ and $M$ such that 
for a branch $\Lambda$  of the inverse of $Z$ we have 
\begin{equation}\label{8n}
\frac{\alpha_1}{|x|}
\leq 
\ell(D\Lambda(x))
\leq 
\left|D\Lambda(x)\right|\leq\frac{\alpha_2}{|x|} 
\quad\text{if} \ x\in\H_{\geq M},
\end{equation}
provided that the derivative exist, which is the case almost everywhere. Since
\begin{equation}\label{8o}
D\Lambda(x)=DZ(\Lambda(x))^{-1}
\end{equation}
this yields that 
\begin{equation}\label{8p}
\alpha_1\leq \frac{|Z(x)|}{|DZ(x)|}\leq \frac{|Z(x)|}{\ell(DZ(x))}\leq \alpha_2
\quad\text{if} \ Z(x)\in\H_{\geq M}.
\end{equation}
We may assume here that $\alpha_1\leq 1\leq \alpha_2$.

The following lemma shows in particular that similar estimates hold for $H$ and~$G$.
Here $M$ is the constant from~\eqref{8p}.

\begin{lemma} \label{lemma1}
Let $0<\sigma<\tau<1$.
Then there exists $C_1,R>0$ with the following properties:

Let $F\in\{Z,H,G\}$, let $x\in\R^3$ and
let $0<r\leq \sigma |F(x)|$ such that $F(x)\in \H_{\geq R}$ and
\begin{equation}\label{8p1}
B(F(x),r)\subset\H_{\geq M}.
\end{equation}
Let $U$ be the component of $F^{-1}(B(F(x),r))$ containing~$x$.
If $F=G$, suppose in addition that $|F(x)|\geq 3|x|$. Then 
\begin{equation}\label{8q}
\esssup_{u\in U} \max \{|F(u)|,|DF(u)|\}
\leq C_1 \essinf_{u\in U} \min \{|F(u)|,\ell(DF(u))\}
\end{equation}
and
\begin{equation}\label{8r}
\diam U\leq 2C_1 \frac{r}{|F(x)|}.
\end{equation}
The condition~\eqref{8p1} is satisfied in particular if $F(x)_3\geq \tau|F(x)|$.
\end{lemma}
\begin{proof}
Suppose first that $F=Z$. 
For $u\in U$ we have $Z(u)\in\H_{\geq M}$ by~\eqref{8p1}
and 
\begin{equation}\label{8r1}
(1-\sigma)|Z(x)|\leq |Z(x)|-r\leq |Z(u)|\leq|Z(x)|+r\leq  (1+\sigma)|Z(x)|.
\end{equation}
Together with~\eqref{8p} we find that
\begin{equation}\label{8s}
\max \{|Z(u)|,|DZ(u)|\} \leq \frac{1}{\alpha_1}|Z(u)|\leq \frac{1+\sigma}{\alpha_1}|Z(x)|
\end{equation}
while
\begin{equation}\label{8t}
\min \{|Z(u)|,\ell(DZ(u))\} \geq \frac{1}{\alpha_2}|Z(u)|\geq \frac{1-\sigma}{\alpha_2}|Z(x)|
\end{equation}
if $DZ(u)$ exists. 
This yields~\eqref{8q} for $F=Z$ with $C_1=(1+\sigma)\alpha_2/((1-\sigma)\alpha_1)$.

It follows from this and the chain rule that~\eqref{8q} also holds for $F=H$, with
a larger constant~$C_1$, since 
$|D\varphi(x)|/\ell(D\varphi(x))$ is constant for all $x$ where $\varphi$ is 
differentiable. In fact, we have $|D\varphi(x)|/\ell(D\varphi(x))=(2+\sqrt{3})/(2-\sqrt{3})$
for all such~$x$.

To deal with the case that $F=G$ we note that 
by~\eqref{1c1} we have $G(u)=Z(u)+u$ for $u\in U$, 
provided $R$ is sufficiently large. Hence
\begin{equation}\label{8u}
\frac12 |DZ(u)|\leq |DZ(u)|-1\leq |DG(u)|\leq |DZ(u)|+1\leq 2|DZ(u)|
\end{equation}
and 
\begin{equation}\label{8v}
\frac12 \ell(DZ(u))\leq \ell(DZ(u))-1\leq \ell(DG(u))\leq \ell(D(u))+1\leq\ell(DZ(u)) 
\end{equation}
for $u\in U$.
By hypothesis, $|G(x)|\geq 3|x|$.
We deduce that $|Z(x)|=|G(x)+x|\geq 2|x|$
and hence $|Z(x)|/2\leq |G(x)|\leq 3|Z(x)|/2$.
It follows that
\begin{equation}\label{8w}
\frac{1-\sigma}{2}|Z(x)|
\leq (1-\sigma)|G(x)|\leq |G(u)|
\leq (1+\sigma)|G(x)|
\leq \frac32(1+\sigma)|Z(x)|
\end{equation}
for $u\in U$.
We deduce 
from the last three inequalities 
that~\eqref{8q} also holds for $F=G$ with some constant~$C_1$.

To prove~\eqref{8r} let $v\in\partial U$ with $|v-x|=\max_{u\in\partial U}|u-x|$.
Let $\Gamma$ be the straight line segment connecting $F(x)$ and $F(v)$
and let $\gamma$ be the curve in the closure of $U$ such that $\Gamma=F(\gamma)$.
It follows from~\eqref{8q} that
\begin{equation}\label{8x}
r = \length( \Gamma)
\geq \length(\gamma) \essinf_{u\in U} \ell(DF(u))
\geq |v-x| \frac{|F(x)|}{C_1}
\geq  \frac{\diam U}{2C_1} |F(x)|.
\end{equation}
Now~\eqref{8r} follows.

To prove that~\eqref{8p1} is satisfied if $F(x)_3\geq \tau|F(x)|$ we note that
if this is the case and $w\in B(F(x),r)$, then
$w_3\geq F(x)_3-r\geq (\tau-\sigma)|F(x)|\geq  (\tau-\sigma)R$.
Thus~\eqref{8p1} follows for large~$R$.
\end{proof}
We apply Lemma~\ref{lemma1} to the function $h_m$ defined by~\eqref{3b1}.
\begin{lemma} \label{lemma1a}
Let $0<\sigma<\tau<1$.
Then there exists $C_1,R>0$ with the following properties:

Let $m\in\N$, let $x\in\R^3$ and
let $U$ be the component of $h_m^{-1}(B(h_m(x),\sigma|h_m(x)|))$ containing~$x$.
Suppose that $h_m(x)_3\geq \tau|h_m(x)|$
and that $|h_j(x)|\geq 3|h_{j-1}(x)|$ and $h_j(x)\in\H_{\geq R}$ for $1\leq j\leq m$.
Then
\begin{equation}\label{8y}
\esssup_{u\in U} \max \left\{\prod_{j=1}^m |h_j(u)|,|Dh_m(u)|\right\}
\leq C_1^m \essinf_{u\in U} \min \left\{\prod_{j=1}^m |h_j(u)|,\ell(Dh_m(u))\right\}.
\end{equation}
\end{lemma} 
\begin{proof} 
The chain rule implies that 
\begin{equation}\label{8z}
\prod_{j=1}^m \ell(DF_j(h_{j-1}(u)))
\leq 
\ell(Dh_m(u))
\leq |Dh_m(u)|\leq \prod_{j=1}^m |DF_j(h_{j-1}(u))| ,
\end{equation}
provided the derivatives exist.
The conclusion follows if we show that 
\begin{equation}\label{9a}
\esssup_{w\in h_{j-1}(U)} \max \{|F_j(w)|,|DF_j(w)|\}
\leq C_1 \essinf_{w\in h_{j-1}(U)} \min \{|F_j(w)|,\ell(DF_j(w))\}
\end{equation}
for $1\leq j\leq m$.
For $j=m$ this follows from~\eqref{8q}, with $F=F_m$ and $r=\sigma|F_m(x)|$.
For $1\leq j\leq m-1$ it would also follow from~\eqref{8q} if
we show that there exists $r\in (0,\sigma |h_j(x)|)$ such that
\begin{equation}\label{9b}
h_j(U)\subset B(h_j(x),r)\subset \H_{\geq M}.
\end{equation}
By~\eqref{8r} we have $h_{m-1}(U)\subset B(h_{m-1}(x),r)$ with $r=\diam h_{m-1}(U)\leq 2C_1\sigma$.
Thus $B(h_{m-1}(x),r)\subset\H_{\geq M}$
if $R$ is sufficiently large. Hence~\eqref{9b} holds for $j=m-1$.
Inductively we see 
that~\eqref{9b} and hence~\eqref{9a} also hold for $1\leq j\leq m-2$.
\end{proof} 

\section{Proof of Theorem~\ref{theorem_dimA}} \label{lower_bound}
\subsection{The upper bound for the dimension} \label{upper_bound}
We may assume that $\rho<3$ since otherwise there is nothing to prove.
Fix a (large) number $R$ and let $\cA_R(f)$ be the set of all $x\in T\cap\H_{\geq R}$
such that $h_m(x)\in\H_{\geq R}$,
$|h_{m}(x)| \geq \exp^{m}(R)$
and $|h_m(x)|\geq 3 |h_{m-1}(x)|$
for all $m\in\N$, 
as well as $h_m(x)\in T$ for all even~$m$.
It follows from~\eqref{8i} and~\eqref{8j} that 
$\cA(f)=\bigcup_{n=0}^\infty f^{-n}(\cA_R(f))$.
Since $f$ is locally bi-Lipschitz, it suffices to show that
$\dim \cA_R(f)\leq \rho$ for all~$L$.

We will do so using Lemma~\ref{lemma10}.
So let $x\in \cA_R(f)$.
It follows from~\eqref{1c5} that if $m$ is even, then $h_m(x)$ is contained in 
the domain $\Omega$ defined by~\eqref{1c4}.
This implies that 
\begin{equation}\label{6a}
|h_m(x)|\leq 2 h_m(x)_3.
\end{equation}
Choose  $t>0$ large enough  and put 
\begin{equation}\label{6a1}
K_m=B\!\left(h_m(x),\frac14 |h_m(x)|\right)
\end{equation}
and let $U_m(x)$ be the component of $h_m^{-1}(K_m)$ that contains~$x$.
We now define $r_m(x)=\dist(x,\partial U_m(x))$ so that $B(x,r_m(x))\subset U_m(x)$.
Let $y_m\in\partial U_m(x)$ with $|y_m-x|=r_m(x)$ and let
$\gamma$ be the straight line segment connecting $x$ and $y_m$.
Then $h_m(\gamma)$ connects $h_m(x)$ with $\partial K_m$.
Lemma~\ref{lemma1a} yields that
\begin{equation}\label{6b}
\frac14 |h_m(x)| \leq \length (h_m(\gamma) )
\leq \length(\gamma) \esssup_{y\in U_m(x)} |Dh_{m}(y)|
\leq r_m(x) C_1^{m} \prod_{j=1}^{m} |h_j(x)|.
\end{equation}
Thus 
\begin{equation}\label{6c}
r_m(x)\geq \frac{C_1^{-m}}{4 \prod_{j=1}^{m-1} |h_j(x)|}.
\end{equation}

If $s=(s_1,s_2)\in S$ and if $T(s)$ intersects $K_m$, then
\begin{equation}\label{6d}
|(h_m(x)_1,h_m(x)_2)-(s_1,s_2)|\leq \frac14 |h_m(x)| +\sqrt{2}\leq |h_m(x)|,
\end{equation}
provided $R$ is chosen sufficiently large.
Suppose that $\rho>1$. Then 
$0<\beta<1$ by~\eqref{def-beta} and 
it follows from~\eqref{7a0} and~\eqref{6d} that
\begin{equation}\label{6e}
\card\!\left\{s\in S\colon T(s)\cap K_m)\neq\emptyset \right\}
\leq \alpha_0 |h_m(x)|^{2-2\beta} .
\end{equation}
But if $\rho=1$ so that $\beta=1$,
then $S=\{(0,0)\}$ and hence $\card S=1$. Thus~\eqref{6e} holds
trivially also in this case, assuming that $\alpha_0\geq 1$.

For each $s$ we can cover $T(s)\cap K_m$  by $M$ cubes of 
sidelength $2$, where 
\begin{equation}\label{6f}
M\leq \frac12 |h_m(x)| +1 \leq |h_m(x)|.
\end{equation}
It follows from the last two inequalities that $T\cap K_m$
can be covered by $N_m(x)$ cubes  of sidelength $2$, where 
\begin{equation}\label{6g}
N_m(x)\leq |h_m(x)|^{3-2\beta} =|h_m(x)|^\rho .
\end{equation}
Let $W$ be the image of such a cube under the inverse 
$\varphi\colon K_m \to U_m(x)$ of the map $h_m\colon U_m(x)\to K_m$.
As these cubes have diameter~$2\sqrt{3}$, Lemma~\ref{lemma1a} yields that
\begin{equation}\label{6h}
\diam W\leq d_m(x) 
\end{equation}
where
\begin{equation}\label{6i}
d_m(x) 
:= 2\sqrt{3} \esssup_{y\in K_m} |D\varphi(y)|
= \frac{2\sqrt{3}}{\essinf\limits_{z\in U_m(x)}\ell(Dh_{m}(z))}
\leq  \frac{2\sqrt{3}C_1^{m}}{\prod_{j=1}^{m} |h_j(x)|}.
\end{equation}
To summarize we see that $B(x,r_m(x))\cap \cA_R(f)$ can be covered by $N_m(x)$ sets of 
diameter at most $d_m(x)$, where $r_m(x)$, $N_m(x)$ and $d_m(x)$ satisfy~\eqref{6c},
\eqref{6g} and~\eqref{6i}, respectively.
In order to apply Lemma~\ref{lemma10}, let $\varepsilon>0$.
We deduce that
\begin{equation}\label{6j}
\begin{aligned}
\frac{N_m(x)d_m(x)^{\rho +\varepsilon}}{r_m(x)^{3}}
&\leq 
|h_m(x)|^\rho 
\left(  \frac{2\sqrt{3}C_1^{m}}{\prod_{j=1}^{m} |h_j(x)|}\right)^{\rho+\varepsilon}
\left( \frac{4 \prod_{j=1}^{m-1} |h_j(x)|}{C_1^{-m}}\right)^{3}
\\  &
= 4^3\cdot  \left(2\sqrt{3}\right)^{\rho+\varepsilon} C_1^{(3+\rho+\varepsilon)m}
|h_m(x)|^{-\varepsilon} \left(\prod_{j=1}^{m-1} |h_j(x)|\right)^{\! 3-\rho-\varepsilon} .
\end{aligned}
\end{equation}
It follows from~\eqref{8m} and~\eqref{eq8m} that the right hand side is less than~$1$ for large~$m$.
Lemma~\ref{lemma10} now yields that $\dim \cA_R(f)\leq \rho+\varepsilon$. Since 
$\varepsilon$ can be taken arbitrarily small, we conclude that $\dim \cA_R(f)\leq \rho$
and hence that $\dim\cA(f)\leq\rho$.

\subsection{Nested sets} \label{nested}
To estimate $\dim \cA(f)$ from below we will construct a subset of $\cA(f)$ as follows.
For $m\in \N$ we construct a collection ${\mathcal E}_m$ of disjoint compact
subsets of $\R^3$ with the property that every element of ${\mathcal E}_{m+1}$ is
contained in a unique element of ${\mathcal E}_m$. Conversely, every element
of ${\mathcal E}_m$ contains at least one element of ${\mathcal E}_{m+1}$.  Put
\begin{equation}\label{3a}
E_m=\bigcup_{V\in\cE_m} V
\quad\text{and}\quad
E=\bigcap^\infty_{m=1} E_m. 
\end{equation}
The construction will me made in such a way that $E\subset \cA(f)$.

McMullen~\cite{McMullen1987} used Lemma~\ref{lemma-frostman} to give a lower bound 
for the dimension of a set $E$ constructed this way in terms of 
$\sup_{V\in\cE_m}\diam V$ and $\inf_{V\in\cE_m}\dens(E_{m+1},V)$.
Adapting some ideas from~\cite{BKS},
we will instead work with bounds for $\diam V$ and $\dens(E_{m+1},V)$
which depend not only on~$m$, but also on~$V$.

We will choose the sets ${\mathcal E}_m$ such that $h_m$ maps the elements
of $\cE_m$ bijectively onto a closed ball of the form 
$K(t)=\overline{B}((0,0,t),t/2)$ for some large~$t$, say $t\geq R$.
Various conditions imposed later will require that $R$ is sufficiently large.
We will also consider the box
\begin{equation}\label{3c2}
Q(t)=\left\{x\in K(t)\colon \frac{t}{6}\leq x_1\leq \frac{t}{4},\;
\frac{t}{6}\leq x_2\leq \frac{t}{4},\;
|x_3-t|\leq \frac{t}{4}\right\} .
\end{equation}
Note that $Q(t)\subset K(t)$.
The advantage of considering $Q(t)$ will be that points in $Q(t)$ have a definite
distance from $\partial K(t)$ and that two beams $T(s)$ and $T(s')$ intersecting 
$Q(t)$ are a certain distance apart.

Put $\eta(r)=\log\log r$ and, for $ 2\leq k\in\N$ and $s\in S$, 
\begin{equation}\label{3d}
p(s,k)=(2s_1,2s_2,k\eta(k)).
\end{equation}

Let $P_H(s,k)$ be the component of $H^{-1}(K(H(p(s,k))_3))$ containing $p(s,k)$.
Here $H(p(s,k))_3$ denotes the third component of $H(p(s,k))$.
Note that since
$\varphi(p(s,k))=(2s_1,2s_2,k\eta(k)-2|s_1|-2|s_2|)$
we have 
$H(p(s,k))=(0,0,\exp(k\eta(k)-2|s_1|-2|s_2|))$ 
so that $H(p(s,k))_3=\exp(k\eta(k)-2|s_1|-2|s_2|))$.
Note that $P_H(s,k)\cap P_H(s',k')=\emptyset$ if $(s,k)\neq (s',k')$, assuming that
$k$ and $k'$ are large enough.

For large $t$ we consider the set $\cU_H(t)$ 
of all $P_H(s,k)$ which are contained in $Q(t)$.
Note that if $P_H(s,k)\in \cU_H(t)$, then, in particular, $p(s,k)\in Q(t)$.
This yields that  $2|s_i|\leq t/4$ for $i\in\{1,2\}$ while $k\eta(k)\geq 3t/4$.
We deduce that 
\begin{equation}\label{3d0}
k\eta(k)\geq 3(|s_1|+|s_2|).
\end{equation}
This yields  that $k\eta(k)-2|s_1|-2|s_2|\geq k\eta(k)/3$
and $|p(s,k)|\leq 5k\eta(k)/3$.
Thus
\begin{equation}\label{3d1}
\begin{aligned}
|H(p(s,k))|
&=\exp(k\eta(k)-2|s_1|-2|s_2|)
\\ &
\geq \exp\!\left( \frac13 k\eta(k)\right)\geq \exp\!\left(\frac15 |p(s,k)|\right).
\end{aligned}
\end{equation}

Analogous definitions are made with $G$ instead of~$H$. Thus we define $P_G(s,k)$
as the component of $G^{-1}(K(G(p(s,k))_3))$ that contains $p(s,k)$.
Note here that is not difficult to see that $G(p(s,k))\in K(G(p(s,k))_3)$
if $k$ is sufficiently large; see Remark~\ref{remark1} below.
Again the $P_G(s,k)$ are pairwise disjoint.
Also,  we define
$\cU_G(t)$ as the set of all $P_G(s,k)$ which are contained in $Q(t)$.

Since $G(p(s,k))_3=k\eta(k)+\exp(k\eta(k))$ we have 
\begin{equation}\label{3d2}
|G(p(s,k))|\geq \exp(k\eta(k))\geq \exp\!\left(\frac35 |p(s,k)|\right).
\end{equation}

It follows from the above definitions 
that if $F\in\{G,H\}$ and $P\in \cU_F(t)$, then $F(P)$ is a ball and
$F\colon P\to F(P)$ is bijective.

For large $t$ and $F\in\{G,H\}$ we put
\begin{equation}\label{3e1}
U_F(t)=\bigcup_{P\in \cU_H(t)} P.
\end{equation}

For some large  $R$ we put $K_0=K(R)$,
\begin{equation}\label{3e}
\cE_0=\{ K_0\}
\quad \text{and} \quad
\cE_1= \cU_H(R) .
\end{equation}
Suppose now that $m\in\N$ and that $\cE_m$ has been defined.
Let $V\in \cE_m$.
Then we have $h_m(V)=K(t_V)$ for some $t_V>R$ and $h_m\colon V\to K(t_V)$ is bijective.
Let $\varphi\colon K(t_V)\to V$ be the inverse of this map.
We put 
\begin{equation}\label{3f}
\cE_{m+1}(V)=
\left\{\varphi(P)\colon P\in \cU_{F_{m+1}}(t_V) \right\} 
=
\begin{cases}
\left\{\varphi(P)\colon P\in \cU_G(t_V) \right\} 
 & \text{if}\ m \ \text{is odd},\\
\left\{\varphi(P)\colon P\in \cU_H(t_V) \right\} 
 & \text{if}\ m \ \text{is even},
\end{cases}
\end{equation}
and
\begin{equation}\label{3g}
\cE_{m+1}=\bigcup_{V\in \cE_m} \cE_{m+1}(V).
\end{equation}

Following~\cite{McMullen1987} we construct a
probability measure $\mu$ supported on $E$.
In order to do so, we define  a sequence $(\mu_m)$ of 
probability measures inductively. Here $\mu_m$ is supported on $E_m$,
and the restriction of $\mu_m$ to an element of $\cE_m$ is
a rescaled Lebesgue measure.

First we define $\mu_0$ as the Lebesgue measure on $K_0$,
rescaled so that $\mu_0(K_0)=1$. 
Thus $\mu_0(A)=\meas(A\cap K_0)/\meas K_0$.
Let $m\geq 0$ and suppose that the measure $\mu_{m}$ has been defined.
To define the measure $\mu_{m+1}$ it suffices to specify $\mu_{m+1}(W)$
for $W \in \cE_{m+1}$.
For such $W$ there exists $V\in\cE_m$ with $W\subset V$.
We put
\begin{equation}\label{3g0}
\mu_{m+1}(W)=\frac{\mu_{m}(W)}{\dens(E_{m+1},V)}.
\end{equation}
Note that for $V\in\cE_m$ we have
\begin{equation}\label{3g1}
\begin{aligned}
\mu_{m+1}(V)
&=\sum_{W\in\cE_{m+1}(V)} \mu_{m+1}(W)
=\frac{1}{\dens(E_{m+1},V)}\sum_{W\in\cE_{m+1}(V)} \mu_{m}(W)
\\ &
=\frac{1}{\dens(E_{m+1},V)}\mu_m(V\cap E_{m+1})
=\mu_m(V).
\end{aligned}
\end{equation}
We conclude that 
\begin{equation}\label{3g2}
\mu_{n}(V) =\mu_m(V)
\quad\text{for}\ n\geq m,\; V\in\cE_m.
\end{equation}
Let $\mu$ be a weak limit of the sequence $(\mu_m)$. Then
$\mu$ is a probability measure supported on $E$ such that 
\begin{equation}\label{3g3}
\mu(V) =\mu_m(V)
\quad\text{for}\ V\in\cE_m.
\end{equation}

\subsection{Estimates for the sets in \texorpdfstring{$\cE_m$}{Em}}\label{estimates_Em}
As already mentioned, we will estimate $\diam V$ and $\dens(E_{m+1},V)$
for $V\in\cE_m$.
\begin{lemma} \label{lemma2}
There exist $C_2,C_3,R>0$ such that if $F\in\{G,H\}$, $t\geq R$
 and $P\in\cU_H(t)$, then
\begin{equation}\label{5a}
\diam P\leq C_2
\end{equation}
and
\begin{equation}\label{5a1}
\meas P\geq C_3.
\end{equation}
\end{lemma}
\begin{proof}
Let $P=P_F(s,k)$, with $s\in S$ and $k\in\N$. Put $q=F(p(s,k))_3$ and $a=(0,0,q)$.
Then $F(P)=B(a,q/2)$.

If $F=H$, then $F(p(s,k))=a$ and thus $|F(p(s,k))|=q$.
If $F=G$, then $F(p(s,k))=(2s_1,2s_2,q)$.
Using~\eqref{3d0} we can deduce that
\begin{equation}\label{5a2}
q=k\eta(k)+\exp(k\eta(k))\geq 20(|s_1|+|s_2|)
\end{equation}
if $R$ and hence $k$ are sufficiently large. Hence 
\begin{equation}\label{5a3}
\frac{9q}{10}\leq q-2(|s_1|+2|s_2|)\leq 
|F(p(s,k))|\leq q+2(|s_1|+|s_2|)\leq \frac{11q}{10}
\end{equation}
as well as 
\begin{equation}\label{5a10}
|F(p(s,k))-a|=|(2s_1,2s_2,0)|\leq \frac{q}{10}
\end{equation}
 for large~$k$.
It follows from the last two inequalities that
\begin{equation}\label{5a11}
F(P)=B\!\left(a,\frac{q}{2}\right)\subset B\!\left(F(p(s,k)),\frac{3q}{5}\right)
\subset B\!\left(F(p(s,k)),\frac23|F(p(s,k))|\right).
\end{equation}
Since $F(p(s,k))_3=q\geq 10|F(p(s,k))|/11$ by~\eqref{5a3} we may apply Lemma~\ref{lemma1} 
with $\tau=10/11$ and $\sigma=2/3$. We deduce from~\eqref{8r} that~\eqref{5a} follows
with $C_2=4C_1/3$.

It also follows from~\eqref{5a10} that
\begin{equation}\label{5b}
\dist(F(p(s,k)),\partial F(P))\geq \frac{2q}{5}.
\end{equation}

Let $\gamma$ be the line segment connecting $p(s,k)$ with the closest point on~$\partial P$.
Thus $\length(\gamma)=\dist(p(s,k),\partial P)$. 
Recalling that $F$ is injective on $P$ we find that
$F\circ\gamma$ is a curve connecting $F(p(s,k))$ with~$\partial F(P)$.
Together with Lemma~\ref{lemma1a} it follows that 
\begin{equation}\label{5b1}
\begin{aligned}
\frac{2q}{5}
&\leq \length( F\circ \gamma)
\leq \esssup_{x\in P} |DF(x)| \cdot \length(\gamma)
\\ &
\leq C_1|F(p(s,k))|\dist(p(s,k),\partial P).
\end{aligned}
\end{equation}
Since $|H(p(s,k))|=q$ we see that the estimate $|F(p(s,k))|\leq 11q/10$ 
given in \eqref{5a3} is valid not only for $F=G$ but trivially also holds for $F=H$.
Thus $\dist(p(s,k),\partial P)\geq 4/(11 C_1)$.
This yields~\eqref{5a1} with $C_3=4^4 \pi/(3\cdot (11C_1)^3)$.
\end{proof}

\begin{remark}\label{remark1}
It follows in particular from~\eqref{5a10} that $F(p(s,k))\in K(F(p(s,k))_3)$ if
$k$ is sufficiently large.
Of course this is clear if $F=H$, but~\eqref{5a10} shows that it also holds
for $F=G$ if $R$ is large enough.
\end{remark}

It follows from~\eqref{3d1}, \eqref{3d2} and~\eqref{5a} that if
$x\in P_F(s,k)\in \cU_F(t)$ for some $t\geq R$ and $F\in\{G,H\}$,
then 
\begin{equation}\label{5d1}
|F(x)|\geq \frac12 |F(p(s,k))|
\geq \exp\!\left(\frac15 |p(s,k)|\right)
\geq \exp\!\left(\frac15 |x|-C_2\right)
\geq 6 \exp\!\left(\frac16 |x|\right),
\end{equation}
provided $R$ and hence $k$ are sufficiently large.
It follows that 
\begin{equation}\label{5d2}
|h_m(x)|\geq 6\exp^m\!\left(\frac{1}{6}|x|\right)
\end{equation}
for $x\in E$, if $R$ is chosen sufficiently large.
Thus $E\subset\cA(f)$ by~\eqref{8j}.
\begin{lemma} \label{lemma3}
There exists $C_4,C_5,C_6,R>0$ such that if $F\in\{G,H\}$ and $t\geq R$, then
\begin{equation}\label{3h1}
\card \cU_F(t)
\geq C_4 \frac{t^{\rho}}{\eta(t)}
\end{equation}
and
\begin{equation}\label{3h2}
\dens\!\left(U_F(t),K(t)\right)
\geq C_5 \frac{t^{\rho-3}}{\eta(t)}.
\end{equation}
If  $x\in Q(t)$, then
\begin{equation}\label{3h3}
\card\{P\in  \cU_F(t)\colon P\cap B(x,\tau)\neq\emptyset\}
\leq
\begin{cases}
C_6(\tau +1) &\text{if}\ \tau\leq t^{\beta},\\
C_6\tau^\rho    &\text{if}\ \tau> t^{\beta},\\
\end{cases}
\end{equation}
with $\beta$ defined by~\eqref{def-beta}.
In particular 
\begin{equation}\label{eq3h3}
	\card \cU_F(t)\leq C_6(2t)^\rho.
\end{equation}
Moreover, if $P,P'\in \cU_F(t)$, then
\begin{equation}\label{3h4}
\dist(P,P')\geq\frac14 \eta(t).
\end{equation}
\end{lemma}
\begin{proof}
Recall that  $B_2((x_1,x_2),r)$ denotes
the planar disk centred at $(x_1,x_2)$ and of radius $r$. Assuming that $R$ is large, it follows from~\eqref{5a} that 
if 
\begin{equation}\label{5e}
s
\in
 B_2\!\left( \frac15(t, t),\frac{t}{31}\right)
\subset
 B_2\!\left( \frac15(t, t),\frac{t}{30}-C_2\right)
\end{equation}
and 
\begin{equation}\label{5f}
\frac{3t}{4}+C_2\leq k\eta(k) \leq \frac{5t}{4}-C_2 ,
\end{equation}
then $P_F(s,k)\subset Q(t)$ and thus $P_F(s,k)\in \cU_F(t)$.
By~\eqref{7a}, 
the cardinality of the set of all $s\in S$ satisfying~\eqref{5e} 
is at least 
\begin{equation}\label{5g}
\frac{1}{64} \left(\frac{\sqrt{2}t}{5}\right)^{-2\beta} \frac{t^2}{31^2}
=\frac{2^{-\beta} }{64\cdot 5^{-2\beta} \cdot 31^2} \cdot t^{2-2\beta},
\end{equation}
provided $t$ is sufficiently large.

To estimate the cardinality of the set of all $k$ satisfying~\eqref{5f} we 
claim that~\eqref{5f} is satisfied if 
\begin{equation}\label{5g1}
\left(\frac{3t}{4}+C_2\right)\frac{1}{\eta(t)-1}
\leq k\leq
\left(\frac{5t}{4}-C_2\right)\frac{1}{\eta(t)}
\end{equation}
To see this note that~\eqref{5g1} implies that 
$\sqrt{t}\leq k\leq t$ and thus $\eta(t)-1\leq\eta(k)\leq \eta(t)$.
Using~\eqref{5a} we can now deduce~\eqref{5f} from~\eqref{5g1}.

The cardinality of the 
set of all $k\in\N$ which satisfy~\eqref{5g1} and hence also~\eqref{5f} is at least
\begin{equation}\label{5h}
\frac13 \frac{t}{\eta(t)},
\end{equation}
provided $t$ is sufficiently large.
Combining the bounds~\eqref{5g} and~\eqref{5h} we find that 
\begin{equation}\label{5i}
\card \cU_F(t)
\geq \frac{2^{-\beta} }{3\cdot 64\cdot 5^{-2\beta} \cdot 31^2} \frac{t^{3-2\beta}}{\eta(t)}
\end{equation}
for large $t$. 
Now~\eqref{3h1} follows from~\eqref{def-beta}.

Combining~\eqref{3h1} with~\eqref{5a1} we obtain~\eqref{3h2} with $C_5=C_3C_4/(6\pi)$.

To prove~\eqref{3h3} let $x\in Q(t)$. Arguments similar to the ones used
to obtain~\eqref{5e} and~\eqref{5f} show that if $P_F(s,k)\in \cU_F(t)$ with 
$P_F(s,k)\cap B(x,\tau)\neq\emptyset$, then
\begin{equation}\label{5j}
s\in 
B_2((x_1,x_2), \tau+C_2)
\end{equation}
and 
\begin{equation}\label{5k}
|k\eta(k)-x_3|\leq \tau+C_2.
\end{equation}
By the definition of~$Q(t)$ we have
\begin{equation}\label{5l}
\frac{t}{6}\leq |(x_1,x_2)|\leq |x| \leq 2t .
\end{equation}
Suppose first that $t^\beta\leq\tau\leq t/24$. 
Then, provided $R$ is large enough,  $\tau+C_2\leq 2\tau$ and hence
$B_2((x_1,x_2), \tau+C_2)\subset B_2((x_1,x_2), 2\tau)$.
Since $2\tau\leq t/12\leq 2|(x_1,x_2)|$ we can use~\eqref{7a} and~\eqref{7a1}
to bound the cardinality of $S\cap B_2((x_1,x_2),2\tau)$.
This is bigger than 
the cardinality $N_1$ of the set of all $s\in S$ satisfying~\eqref{5j}. Thus
\begin{equation}\label{5m}
N_1\leq 324 |(x_1,x_2)|^{-2\beta} (2\tau)^2
\leq 324\cdot 6^{2\beta}\cdot  4 t^{-2\beta}\tau^2 .
\end{equation}
Suppose next that $\tau\geq t/24$. Then 
$N_1\leq \alpha_0 (2\tau)^{2-2\beta}\leq
\alpha_0 2^{2-2\beta} 24^{2\beta} t^{-2\beta}\tau^2$ by~\eqref{7a0}.
Suppose finally that $\tau\leq t^\beta$.
Then $\tau+C_2\leq t^\beta+C_2\leq 6^\beta|x|^\beta+C_2\leq 8|x|^\beta$ for large $R$
by~\eqref{5l}. Now~\eqref{7b} yields that $N_1\leq 324$. 
Altogether it follows that there exists $C_6>0$ such that 
\begin{equation}\label{5n2}
4 N_1\leq 
\begin{cases}
 C_6 &\text{if}\ \tau\leq t^{\beta},\\
\displaystyle C_6 t^{-2\beta}\tau^2   &\text{if}\ t^{\beta}\leq\tau .\\
\end{cases}
\end{equation}

For $s\in S$ we now consider the cardinality $N_2(s)$ of the set of all 
$k$ such that $P_F(s,k)$ intersects $B(x,\tau)$.
Since 
\begin{equation}\label{5m1}
(k+1)\eta(k+1)-k\eta(k)\geq \eta(k)\geq \eta(t)-1\geq \frac12 \eta(t)
\end{equation}
if $P_F(s,k)\in \cU_F(t)$ we deduce from~\eqref{5a} that
\begin{equation}\label{5n}
N_2(s)\leq N_2:=\frac{4(\tau+C_2)}{\eta(t)}  +1. 
\end{equation}
Combining the above estimates we obtain 
\begin{equation}\label{5n1}
\card\{P\in  \cU_F(t)\colon P\cap B(x,\tau)\neq\emptyset\}
\leq N_1 N_2.
\end{equation}

Suppose that $\tau\leq t^\beta$. 
Noting that~\eqref{5n} yields that $N_2\leq \tau+2$ we obtain~\eqref{3h3}
from~\eqref{5n1} and~\eqref{5n2}.

Suppose next that $\tau> t^\beta$.
Since $Q(t)\subset K(t)\subset B(x,2t)$ we may assume that $\tau\leq 2t$.
By~\eqref{5n} we have $N_2\leq \tau$ for large $t$.
Together with~\eqref{5n1} and~\eqref{5n2} this yields that 
\begin{equation}\label{5n3}
\card\{P\in  \cU_F(t)\colon P\cap B(x,\tau)\neq\emptyset\}
\leq \frac14 C_6 t^{-2\beta} \tau^3
=\frac14 C_6 \left(\frac{\tau}{t}\right)^{2\beta} \tau^\rho\leq C_6 \tau^\rho.
\end{equation}
So we obtain~\eqref{3h3} also in this case.
Since, as already mentioned, $Q(t)\subset B(x,2t)$, this also yields~\eqref{eq3h3}.

Finally, \eqref{3h4} follows from~\eqref{5a} and~\eqref{5m1}, provided
$R$ is large enough.
\end{proof}

Let $x\in E$. Then $x\in E_m$ for all $m\in\N$ and thus there exists 
$V_m(x)\in\cE_m$ such that $x\in V_m(x)$.
Recall that by the definition of $\cE_m$ we have 
\begin{equation}\label{eq3h}h_m(V_m(x))=K(t_m),\end{equation} for some $t_m>R$ satisfying
$t_m\to\infty$ as $m\to\infty$.
It is easy to see that  $t_m/2\leq |h_m(x)|\leq 2 t_m$.
\begin{lemma} \label{lemma4}
There exists $C_7>0$ such that 
\begin{equation}\label{3i}
\dens(E_{m+1},V_m(x))\geq C_7^{m}  \frac{|h_m(x)|^{\rho-3}}{\eta(|h_m(x)|)}.
\end{equation}
\end{lemma}
\begin{proof}
With $h_m(V_m(x))=K(t_m)$
it follows from Lemma~\ref{lemma-dens}, Lemma~\ref{lemma1} and~\eqref{3h2} that
\begin{equation}\label{5o}
\dens(E_{m+1},V_m(x)) \geq C_1^{-3m} \dens(U_{F_{m+1}}(t_m),K(t_m)))
 \geq C_1^{-3m}  C_5 \frac{t_m^{\rho-3}}{\eta(t_m)}.
\end{equation}
The conclusion now follows since $t_m\leq 2|h_m(x)|$
and $\eta(2t_m)\leq 2\eta(t_m)$.
\end{proof}

\begin{lemma} \label{lemma5}
There exists $C_8>0$ such that if $W\in\cE_{m+1}(V_m(x))$, then
\begin{equation}\label{3k}
C_8^{-m}\leq \diam W \cdot \prod_{j=1}^{m}|h_j(x)| \leq C_8^m .
\end{equation}
\end{lemma}
In particular, it follows from Lemma~\ref{lemma5} that
\begin{equation}\label{3j}
\frac{C_8^{-m}}{\prod_{j=1}^{m}|h_j(x)|}\leq
\diam V_{m+1}(x) \leq  \frac{C_8^{m}}{\prod_{j=1}^{m}|h_j(x)|} .
\end{equation}
\begin{proof}[Proof of Lemma~\ref{lemma5}]
We have $h_m(V_m(x))=K(t_m)$. Let $\varphi\colon K(t_m)\to V_m(x)$ be the inverse function 
of $h_m\colon V_m(x)\to K(t_m)$. Each $W\in \cE_{m+1}(V_m(x))$ is of the form
$W=\varphi(P)$ for some $P\in\cU_{F_{m+1}}(t_m)$.
It follows that 
\begin{equation}\label{5p}
\diam W\leq \esssup_{y\in K(t_m)}|D\varphi(y)|\cdot \diam P
=\frac{\diam P}{\essinf\limits_{w\in V_m(x)}\ell(Dh_m(w))} .
\end{equation}
Lemma~\ref{lemma1a} and~\eqref{5a} now yield that
\begin{equation}\label{5u}
\diam W\leq  \frac{C_2C_1^m}{\prod_{j=1}^m |h_{j}(x)|},
\end{equation}
from which the right inequality in~\eqref{3k} follows
for a suitable $C_8>0$.

To prove the left inequality we proceed analogously, considering the measure of~$W$
instead of its diameter. We have
\begin{equation}\label{5p1}
\meas W\geq \essinf_{y\in K(t_m)}\ell(D\varphi(y))^3\cdot \meas P
=\frac{\meas P}{\esssup\limits_{w\in V_m(x)}|Dh_m(w)|^3} .
\end{equation}
Lemma~\ref{lemma1a} and~\eqref{5a1} now imply that
\begin{equation}\label{5u1}
\meas W\geq  \frac{C_3C_1^{3m}}{\prod_{j=1}^m |h_{j}(x)|^3}.
\end{equation}
Adjusting the value of $C_8$ if necessary we obtain the left inequality in~\eqref{3k}.
\end{proof}
\begin{lemma} \label{lemma6}
Let $W,W'\in \cE_{m+1}(V_{m}(x))$. Then
\begin{equation}\label{3l}
\dist(W,W')\geq \diam W.
\end{equation}
\end{lemma}
\begin{proof}
Let $P=h_m(W)$ and $P'=h_m(W')$. With $h_m(V_m(x))=K(t_m)$ we then have 
$P,P'\in \cU_{F_{m+1}}(t_m)$. 
Let $\gamma$ be a straight line segment connecting $W$ and~$W'$
such that $\length(\gamma)=\dist(W,W')$.
First we assume that $\gamma\subset V_m(x)$.
It then follows from Lemma~\ref{lemma5} and the arguments in its proof
that
\begin{equation}\label{5v}
\begin{aligned}
\dist(P,P')
&\leq \length (h_m(\gamma)) \leq \dist(W,W') \esssup_{w\in V_{m}(x)} |Dh_m(w)|
\\ &
\leq \dist(W,W')  C_1^{m} \prod_{j=1}^{m}|h_j(x)|  
%\\ &
\leq C_1^{m} C_8^{m} \frac{\dist(W,W')}{\diam W} .
\end{aligned}
\end{equation}
Using~\eqref{3h4} we deduce that
\begin{equation}\label{5w}
\frac{ \dist(W,W')}{\diam W}  \geq  C_1^{-m} C_8^{-m}\frac12 \eta\!\left(\frac12 h_m(x)\right).
\end{equation}
It now follows from~\eqref{5d2} that if $R$ is chosen sufficiently large, 
then the right hand side of~\eqref{5w} is at least $1$ for all~$m$.

Suppose now that  $\gamma\not\subset V_m(x)$.
Then $\gamma$ contains a subsegment $\gamma'$ which connects $W$ with $\partial V_m(x)$.
Thus $h_m(\gamma')$ connects $P$ with $\partial K(t_m)$. Since $P\subset Q(t_m)$,
we have $\length(h_m(\gamma'))\geq (2-\sqrt{3})t_m/4$. Estimating the length
of $h_m(\gamma')$ 
as in~\eqref{5v} we obtain
\begin{equation}\label{5w1}
\frac{ \dist(W,W')}{\diam W}  \geq  C_1^{1-m} C_8^{1-m}\frac{(2-\sqrt{3})t_m}{4}
\end{equation}
instead of~\eqref{5w}. Again the conclusion follows if $R$ is chosen
sufficiently large.
\end{proof}
\subsection{Proof of the lower bound}
We may assume that $\rho>1$. 
In order to apply Lemma~\ref{lemma-frostman}, let $x\in E$ and $r>0$. 
For $m\in\N$ we put 
\begin{equation}\label{4a0}
d_m(x)=\diam V_m(x)
\quad\text{and}\quad
\Delta_m(x)=\dens(E_{m+1},V_m(x)).
\end{equation}
We choose $m\in\N$ such that 
\begin{equation}\label{4a}
d_{m+1}(x)\leq r< d_m(x).
\end{equation}

It follows from Lemma~\ref{lemma6} that
\begin{equation}\label{4b}
\mu(B(x,r))
=\mu(B(x,r)\cap V_m(x)) .
\end{equation}
Let 
\begin{equation}\label{4c}
\cY_m=\{W\in \cE_{m+1}(V_m(x))\colon W\cap B(x,r)\neq\emptyset\}.
\end{equation}
Then 
\begin{equation}\label{4d}
\mu(B(x,r))
\leq \sum_{W\in \cY_m} \mu(W)
= \sum_{W\in \cY_m} \mu_{m+1}(W)
=  \sum_{W\in \cY_m} \frac{1}{\prod_{j=0}^m \Delta_j(x)} \frac{\meas W}{\meas K_0} .
\end{equation}
Using Lemma~\ref{lemma4} and Lemma~\ref{lemma5}
we find that
\begin{equation}\label{4e}
\begin{aligned}
\mu(B(x,r))
&\leq 
\frac{1}{\meas K_0}
\card \cY_m \cdot  
\prod_{j=1}^m C_7^{-j} \frac{\eta(|h_j(x)|)}{|h_j(x)|^{\rho-3}}
\cdot \left( 2 \frac{C_8^m}{\prod_{j=1}^{m}|h_j(x)|}\right)^3
\\ &
= 
\frac{8 C_7^{-m(m+1)/2} C_8^{3m} }{\meas K_0}
\card \cY_m \cdot  
\prod_{j=1}^m\frac{\eta(|h_j(x)|)}{|h_j(x)|^{\rho}} .
\end{aligned}
\end{equation}
To estimate $\card\cY_m$ we note that 
\begin{equation}\label{4i}
h_m(B(x,r))\subset B\!\left(h_m(x), r \esssup_{w\in B(x,r)} |Dh_m(w)|\right).
\end{equation}
If $B(x,r)\subset  V_m(x)$ or, equivalently, $h_m(B(x,r))\subset K(t_m)$, 
with $t_m$ defined by~\eqref{eq3h},
then, by Lemma~\ref{lemma1a},
\begin{equation}\label{4j}
 r \esssup_{w\in B(x,r)} |Dh_m(w)| \leq \tau_m :=r C_1^m\prod_{j=1}^m |h_j(x)|.
\end{equation}
Since $h_m(x)\in Q(t_m)$, this is the case in particular if
$\tau_m\leq (2-\sqrt{3})t_m/4$. It follows that 
\begin{equation}\label{4g1}
\card\cY_m \leq 
\card\{P\in  \cU_H(t_m)\colon P\cap B(h_m(x),\tau_m)\neq\emptyset\}
\quad\text{if}\ \tau_m\leq\frac{(2-\sqrt{3})t_m}{4}.
\end{equation}

Let $0<\varepsilon<\rho-1$.
If $\tau_m\leq 1$, we deduce from~\eqref{3h3} that $\card \cY_m\leq 2C_6$. Since $|h_j(x)|\to\infty$ we have that $\eta(|h_j(x)|)\leq |h_j(x)|^{\varepsilon/2}$, for large $j$.
Together with~\eqref{4e}, \eqref{3j} and~\eqref{eq8m} we find that 
\begin{equation}\label{4g2}
\begin{aligned}
\mu(B(x,r))
&\leq 
\frac{16C_6 C_7^{-m(m+1)/2} C_8^{3m} }{\meas K_0}
\prod_{j=1}^m\frac{\eta(|h_j(x)|)}{|h_j(x)|^{\rho}} 
\\ &
\leq C_8^{-m(\rho-\varepsilon)} \prod_{j=1}^m\frac{1}{|h_j(x)|^{\rho-\varepsilon}} 
\leq d_{m+1}(x)^{\rho-\varepsilon}\leq r^{\rho-\varepsilon}
\end{aligned}
\end{equation}
for large $m$.

If $1<\tau_m\leq t_m^{\beta}$,
then~\eqref{3h3} yields that $\card \cY_m\leq 2C_6 \tau_m$.
Together with~\eqref{4e}, \eqref{4j}, \eqref{3j}  and~\eqref{eq8m} we deduce that
\begin{equation}\label{4g3}
\begin{aligned}
\mu(B(x,r))
&\leq 
\frac{16C_6 C_7^{-m(m+1)/2} C_8^{3m} C_1^m}{\meas K_0}
r \prod_{j=1}^m\frac{\eta(|h_j(x)|)}{|h_j(x)|^{\rho-1}} 
\\ &
\leq C_8^{-m(\rho-1-\varepsilon)} r \prod_{j=1}^m\frac{1}{|h_j(x)|^{\rho-1-\varepsilon}} 
%\\ &
\leq r d_{m+1}(x)^{\rho-1-\varepsilon}\leq r^{\rho-\varepsilon}
\end{aligned}
\end{equation}
for large $m$.

Suppose next that $t_m^{\beta}<\tau_m\leq (2-\sqrt{3})t_m/4$.
It then follows from~\eqref{3h3} that $\card \cY_m\leq C_6 \tau_m^\rho$.
Together with~\eqref{4e}, \eqref{4j}, \eqref{3j}  and~\eqref{eq8m} this yields that 
\begin{equation}\label{4g4}
\begin{aligned}
\mu(B(x,r))
&\leq 
\frac{8C_6 C_7^{-m(m+1)/2} C_8^{3m} C_1^{ms}}{\meas K_0} r^\rho \prod_{j=1}^m\eta(|h_j(x)|)
\\ &
\leq r^\rho d_{m+1}(x)^{-\varepsilon}\leq r^{\rho-\varepsilon}
\end{aligned}
\end{equation}
for large $m$.

Suppose finally that $\tau_m>(2-\sqrt{3})t_m/4\geq t_m/15$. It follows from \eqref{eq3h3} that  
 $\card \cY_m\leq \card\cU_H(t_m)\leq C_6(2t_m)^\rho\leq C_6 30^\rho\tau_m^\rho$.
As in~\eqref{4g4} we find that $\mu(B(x,r))\leq r^{\rho-\varepsilon}$ for large~$m$.

Thus in all cases we have $\mu(B(x,r))\leq r^{\rho-\varepsilon}$ if $m$ is 
sufficiently large, and hence if $r$ is sufficiently small.
Lemma~\ref{lemma-frostman} yields that $\dim E\geq \rho-\varepsilon$.
Since $E\subset \cA(f)$ and since $\varepsilon>0$ can be chosen arbitrarily small, 
we conclude that $\dim\cA(f)\geq \rho$.
This completes the proof of the lower bound for $\dim\cA(f)$ and hence the
proof of Theorem~\ref{theorem_dimA}.\qed

\begin{remark}\label{dimI}
The function $f$ considered by Rempe and Stallard~\cite{Rempe2010}
that satisfies $\dim\cI(f)=1$ behaves 
like $\exp\exp(z)$ in a half-strip and is bounded outside  this half-strip.
This suggests that in order to construct an example of a quasiregular map
for which $\dim\cI(f)$ is small one would look for a function $f$ which behaves
like $Z(Z(x))$ in a half-beam and which is bounded outside this half-beam.
We have been unable to construct such a function. 
	
However, the methods of~\cite[\S~4]{BD} indicate that such a function would 
have an invariant Cantor set of dimension greater than $d-1$. 
Arguments similar to those in~\cite{BP} now suggest that $\dim\cI(f)\geq d-1$ 
for such a function~$f$.
Perhaps we have $\dim\cI(f)\geq d-1$ and $\dim\cJ(f)\geq d-1$ 
for all quasiregular maps $f\colon\R^d\to\R^d$ of transcendental type.
\end{remark}

\section{Proof of Theorem \ref{theorem_dimJ}}\label{proof_theorem_dimJ}
We are going to need the following lemma which is implicit in the proof
of \cite[Theorem 1.2]{BFN}. We include a sketch of the proof here for  convenience.
\begin{lemma}\label{lemma5.1}
Let $f\colon\mathbb{R}^d\to\mathbb{R}^d$ be a quasiregular map and
$x_1\in \mathcal{A}(f)\setminus\mathcal{J}(f)$. Let $r>0$ be such that
$B(x_1,4r)\cap \mathcal{J}(f)=\emptyset$ and $x_2\in B(x_1,r)$. 
Then there is a constant $c>0$, depending only on $d$ and $r$, such that for all 
large enough $k\in\mathbb{N}$ we have 
\[
\log\frac{\log|f^k(x_2)|}{\log|f^k(x_1)|}\leq c\left(K_I(f)K_O(f)\right)^{k/(d-1)} .
\]
\end{lemma}
\begin{proof}
Since $B(x_1,4r)\cap \mathcal{J}(f)=\emptyset$, the iterates $f^k$ of $f$ omit
a set of positive capacity in $B(x_1,r)$. Let $R>0$ be such that
$B(0,R)\cap \mathcal{J}(f)\not=\emptyset$. For all large enough $k$ we will
have that $R_1:=|f^k(x_1)|\geq R$. Put $R_2:=|f^k(x_2)|$. We may assume
that $R_2>R_1$ since otherwise there is nothing to prove.
	
For any such $k$ now consider the sets $X_1:=\{x\in B(x_1,2r)\colon |f^k(x)|\leq R_1\}$
and $X_2:=\{x\in B(x_1,2r)\colon |f^k(x)|\geq R_2\}$. Denote by $Y_j$ the component of
$X_j$ that contains $x_j$, for $j=1,2$. The maximum principle now implies that $Y_2$
connects $x_2$ to $\partial B(x_1,2r)$. Moreover, $Y_1$  connects $x_1$ to 
$\partial B(x_1,2r)$. Indeed, if this was not the case then  $Y_1$ would be
compactly contained in $B(x_1,2r)$ and thus it would be a connected component
of $(f^k)^{-1}\!\left(\overline{B}(0,R_1)\right)$. This would imply that
\[
B(0,R)\subset \overline{B}(0,R_1)=f^k\!\left(Y_1\right)\subset f^k\!\left(B(x_1,2r)\right).
\]
Since $B(0,R)\cap \mathcal{J}(f)\not= \emptyset$ by complete invariance of
the Julia set, this yields that $B(x_1,2r)\cap \mathcal{J}(f)\not=\emptyset$. 
This is a contradiction since $B(x_1,4r)\cap \mathcal{J}(f)=\emptyset$.
	
Let $\Gamma=\Delta\!\left(X_1,X_2;B(x_1,2r)\right)$ and
$\Gamma_1= \Delta\!\left(Y_1,Y_2;B(x_1,2r)\right)$. We now argue as in the proof
of \cite[Theorem 1.2]{BFN} in order to estimate the modulus $M(\Gamma)$ of the path
family $\Gamma$. We find that
\[
c_1\leq M\!\left(\Gamma_1\right) \leq M(\Gamma)
\leq c_2 K_O(f^k)K_I(f^k)\left(\log\frac{\log |f^k(x_2)|}{\log|f^k(x_1)|}\right)^{1-d},
\]
where the positive constants $c_1$ and $c_2$ depend only on $r$ and~$d$. After
rearranging and using the fact that $K_I(f^k)\leq K_I(f)^k$ and
$K_O(f^k)\leq K_O(f)^k$ this gives the desired inequality.
\end{proof}
\begin{proof}[Proof of Theorem \ref{theorem_dimJ}]
Towards a contradiction assume that there is a function $f$ satisfying \eqref{0d}
with $\dim\mathcal{J}(f)<1$. Then $\mathcal{J}(f)$ is totally disconnected and
$\mathcal{J}(f)=\partial\mathcal{A}(f)$  by \cite[Theorem 1.2]{BFN}.
Hence, the complement of $\mathcal{J}(f)$ comprises of one unbounded component $U$ 
which is contained in $\cA(f)$.
	
Choose $\alpha>(K_IK_O)^{1/(d-1)}$. Then $\alpha> 1$. Let also $x_1\in U$
and, using~\eqref{0d},
choose $R>|x_1|$  so large that
\begin{equation}\label{eq4}
M(r,f)>r
\quad\text{and}\quad
\log M(r,f)>(\log r)^\alpha \quad \text{for all}\ r\geq R.
\end{equation}
Since $x_1\in U\subset \cA(f)$ there exists $L\in\mathbb{N}$ such that 
\begin{equation}\label{eq5}
\left|f^{k+L}(x_1)\right|\geq M^{k+1}(R,f)
\end{equation}
for all $k\in\N$.
Since $|x_1|<R$ and hence $|f^k(x_1)|\leq M^{k}(R,f)$ this implies that 
\begin{equation}\label{eq6}
\left|f^{k+L}(x_1)\right|\geq M\!\left(M^{k}(R,f),f\right)\geq M\!\left(|f^k(x_1)|,f\right)
\end{equation}
for all $k\in\N$.
	
Choose an arc $\gamma$ in $U$ connecting $x_1$ and $f^{L}(x_1)$. 
Cover this arc by $m$ balls $B(x_i,r_i)$ so that $x_m=f^{L}(x_1)$ 
and such that $x_i\in\gamma$ and $B(x_i,4r_i)\cap \mathcal{J}(f)=\emptyset$ for
$i=1,\dots m$.
Applying now Lemma \ref{lemma5.1} $2m-2$ times we obtain 
\begin{equation}\label{eq1}
\log\frac{\log|f^{k+L}(x_1)|}{\log|f^k(x_1)|}\leq C\left(K_I(f)K_O(f)\right)^{k/(d-1)}
\end{equation} 
for all large enough $k$,
where $C$ is a constant that depends on $m, d$ and the radii $r_i$. 
	
Combining this with \eqref{eq6} yields that
\[
\log\frac{\log M(|f^k(x_1)|,f)}{\log|f^k(x_1)|}\leq C\left(K_I(f)K_O(f)\right)^{k/(d-1)}.
\]
Since $|f^k(x_1)|\to\infty$ as $k\to \infty$ we may apply the second inequality
in~\eqref{eq4} for 
$r=|f^k(x_1)|$ if $k$ is sufficiently large. 
It thus follows from the last equation that
\begin{equation}\label{eq7}
(\alpha-1)\log\log|f^k(x_1)|\leq C\left(K_I(f)K_O(f)\right)^{k/(d-1)}
\end{equation} 
for large~$k$.

It follows from~\eqref{eq4} that
\[
\log M^{k}(R,f)\geq \left(\log R\right)^{\alpha^{k}}
\]
for all $k\in\N$.
Together with~\eqref{eq5} we thus have 
\[
\log\log|f^k(x_1)|\geq \log \log M^{k+1-L}(R,f)\geq \alpha^{k+1-L}\log R
\]
for large~$k$.
Since
$\alpha>(K_IK_O)^{1/(d-1)}$ this contradicts~\eqref{eq7} for large~$k$.
\end{proof}

\end{document}